\documentclass[11pt,reqno]{amsart}

\usepackage[leqno]{amsmath}
\usepackage{amsthm}
\usepackage{amsfonts, tikz-cd}
\usepackage{amssymb}
\usepackage{eucal}
\usepackage[all]{xy}
\usepackage{mathtools}
\usepackage{stmaryrd}
\usepackage{csquotes}
\usepackage{enumitem}
\usepackage{upgreek}

\setenumerate{itemsep=2pt,parsep=2pt,before={\parskip=2pt}}

\usepackage[colorlinks=true,hyperindex, linkcolor=magenta, pagebackref=false, citecolor=cyan,pdfpagelabels]{hyperref}

\usepackage{xspace}
\usepackage{epsfig,epic,eepic,latexsym,color}
\usepackage[all]{xy}
\usepackage{comment} 

\usepackage[english]{babel}
\usepackage{latexsym}

\newcommand{\nc}{\newcommand}
\nc{\rnc}{\renewcommand}
\rnc{\P}{\mathbf P}
\nc{\R}{\mathbf R}
\rnc{\rm}{\mathrm}

\nc{\C}{\mathbf C}
\nc{\Q}{\mathbf Q}
\nc{\Z}{\mathbf Z}
\nc{\N}{\mathbf N}
\nc{\A}{\mathbf A}

\nc{\an}{\operatorname{an}}

\nc{\htt}{\operatorname{ht}}
\nc{\Nm}{\operatorname{Nm}}
\nc{\Ker}{\operatorname{Ker}}
\nc{\mmod}{\operatorname{mod}}
\nc{\End}{\operatorname{End}}
\nc{\Aut}{\operatorname{Aut}}
\nc{\cont}{\text{cont}}
\nc{\sep}{\text{sep}}
\nc{\Hom}{\mathrm{Hom}}
\nc{\Gal}{\mathrm{Gal}}
\nc{\Spec}{\operatorname{Spec}}
\nc{\Spv}{\operatorname{Spv}}
\nc{\supp}{\text{supp}}
\nc{\rad}{\operatorname{rad}}
\nc{\cal}{\mathcal}

\nc{\RZ}{\operatorname{RZ}}

\rnc{\t}{\tau}
\nc{\mm}{\pmb{\mu}}
\rnc{\a}{\alpha}
\nc{\n}{\mathfrak n}
\nc{\m}{\mathfrak m}
\nc{\mfs}{\mathfrak s}

\nc{\p}{\mathfrak p}
\nc{\q}{\mathfrak q}

\nc{\Sym}{\operatorname{Sym}}
\nc{\codim}{\operatorname{codim}}
\nc{\rk}{\operatorname{rk}}
\nc{\GL}{\operatorname{GL}}
\nc{\SL}{\operatorname{SL}}
\nc{\Lie}{\operatorname{Lie}}
\nc{\Ind}{\operatorname{Ind}}
\nc{\Div}{\underline{Div}}
\nc{\Pic}{\mathbf{Pic}}
\nc{\uPic}{\underline{ \mathbf{Pic}}}
\nc{\rH}{\mathrm{H}}
\nc{\Spf}{\operatorname{Spf}}
\nc{\Frac}{\operatorname{Frac}}
\nc{\colim}{\operatorname{colim}}
\nc{\Spa}{\operatorname{Spa}}
\nc{\Tor}{\operatorname{Tor}}

\rnc{\an}{\operatorname{an}}

\nc{\xr}{\xrightarrow}

\nc{\eps}{\epsilon}
\nc{\ov}{\overline}
\nc{\ud}{\underline}
\nc{\wdh}{\widehat}

\nc{\F}{\mathcal F}
\nc{\G}{\mathcal G}
\nc{\E}{\mathcal E}

\nc{\X}{\mathfrak X}
\nc{\sZ}{\mathfrak Z}
\nc{\Y}{\mathfrak Y}
\nc{\T}{\mathfrak T}
\nc{\sU}{\mathfrak U}
\nc{\V}{\mathfrak V}

\nc{\LL}{\mathcal{L}}
\rnc{\S}{\mathfrak S}

\nc{\ra}{\rangle}

\nc{\os}{\overset}

\rnc{\O}{\mathcal O}
\nc{\J}{\mathcal J}
\theoremstyle{definition}

\newtheorem{thm}{Theorem}[section]
\newtheorem{lemma}[thm]{Lemma}
\newtheorem{defn}[thm]{Definition}
\newtheorem{claim}{Claim}
\newtheorem{claim1}{Claim}

\newtheorem{Set-up}[thm]{Set-up}

\newtheorem{rmk}[thm]{Remark}

\newtheorem{cor}[thm]{Corollary}

\begin{document}
\bibliographystyle{halpha-abbrv}
\title{Sheafiness of Strongy Rigid-Noetherian Huber Pairs}
\author{Bogdan Zavyalov}
\maketitle

\begin{abstract}
We show that any strongly rigid-noetherian Huber ring $A$ is sheafy. In particular, we positively answer Problem~$31$ in the Nonarchimedean Scottish Book.
\end{abstract}
\section{Introduction}

In the paper \cite{H1}, Huber defined the notion of an adic spectrum $\Spa(A, A^+)$ for a Huber pair $(A, A^+)$. One of the main nuisances of this theory is that the structure presheaf $\O_{(A, A^+)}$ is not always a sheaf on $\Spa(A, A^+)$ (see \cite[Example after Proposition 1.6]{H1}). However, Huber showed that $\O_{(A, A^+)}$ is a sheaf in two important cases: if $A$ is a strongly noetherian Tate ring; and if $A$ has a noetherian ring of definition. The former case was later generalized in \cite{KedAr} to the strongly noetherian analytic case. Huber gave different arguments for the two cases. In the former case his argument is very close in the spirit to Raynaud's theory of admissible blow-ups and ``generic fibers''; in the latter case he was able to adapt the Tate's proof of sheafiness of $\O_A$ in the rigid geometry, this argument is based on some analytic considerations. \smallskip

The two mentioned above examples cover adic spaces that come from rigid spaces or noetherian formal schemes. However, one important disadvantage of these results is that they do not cover formal schemes that are (locally) topologically finitely presented over $\O_{\C_p}$ as the ring $\O_{\C_p}$ is not noetherian. In contrast, there is a good theory of formal schemes over $\O_{\C_p}$ developed, for example, in \cite{B}, and significantly generalized in \cite{FujKato}. \smallskip

David Hansen proposed a question in the Nonarchimedean Scottish Book if any complete, universally topologically rigid-noetherian ring $A$ (see Definition~\ref{defn:rigid-noetherian}) is sheafy. The main reason why Huber's proof in the case of a noetherian ring of definition does not work in this more general setup is that Huber needs to use certain finiteness results from \cite{EGA3} that require the noetherian hypotheses. Our main new idea is to use results from the recent book \cite{FujKato} in place of \cite{EGA3} to make Huber's argument work in a bigger generality. \smallskip

Based on this approach, we are able to show that any Huber ring $A$ with a topologically universally rigid-noetherian ring of definition is sheafy. This unifies the two cases done by Huber and the case done by Kedlaya. Moreover, it provides a new proof in the strongly noetherian analytic case that does not use any analytic considerations.

\begin{thm}\label{thm:main-intro}(Theorem~\ref{thm:main}) Let $(A, A^+)$ be a strongly rigid-noetherian Huber pair (see Definition~\ref{defn:rigid-noetherian}). Then the structure presheaf $\O_X$ is a sheaf of topological rings on $X=\Spa(A, A^+)$. Furthermore, $\rm{H}^i(U, \O_X)=0$ for any rational subdomain $U\subset X$ and  $i\geq 1$.
\end{thm}

\begin{cor}(Lemma~\ref{lemma:sheafy-formal} and Theorem~\ref{thm:main-intro}) Let $k^+$ be a complete microbial valuation ring, and $A$ a topologically finite type $k^+$-algebra. Then the structure presheaf $\O_X$ is a sheaf of topological rings on $X=\Spa(A, A)$. Furthermore, $\rm{H}^i(U, \O_X)=0$ for any rational subdomain $U\subset X$ and  $i\geq 1$.
\end{cor}

\begin{cor}(\cite[Theorem 1.2.11]{KedAr}, Lemma~\ref{lemma:analytic-sheafy} and Theorem~\ref{thm:main-intro}) Let $(A, A^+)$ be a Huber pair with an analytic, strongly noetherian $A$. Then the structure presheaf $\O_X$ is a sheaf of topological rings on $X=\Spa(A, A^+)$. Furthermore, $\rm{H}^i(U, \O_X)=0$ for any rational subdomain $U\subset X$ and  $i\geq 1$.
\end{cor}

Let us now discuss new complications in the proof of Theorem~\ref{thm:main-intro} that do not appear in the classical proof when $A$ has a noetherian ring of definition. \smallskip

The first complication is that the finiteness results that Huber uses in his proof are not known in this generality; instead we use theory of the FP-approximated sheaves (see Appendix~\ref{section:FP}) to get finiteness only up to some torsion modules. These results were announced in \cite[Appendix C to Chapter I]{FujKato} but the proofs will appear only in their upcoming work. The second problem is that even if we try to work in a less general situation (i.e. topologically universally adhesive rings) where the finiteness results are known, there is a problem due to the issue that certain morphisms/sheaves are only of finite type and not of finite presentation. The finiteness results (probably) hold only under the finite presentation assumption. This issue can also be elegantly resolved by using the theory of FP-approximated sheaves. \smallskip

There are other sheafiness results that are somehow orthogonal to the result of this paper. For instance, Scholze showed sheafiness of perfectoid algebras in \cite{Sch0}, Buzzard and Verberkmoes generalized it to any stably uniform Tate ring in \cite{Buzzard}, and recently Hansen and Kedlaya \cite{KedHan} gave new examples of sheafy rings by verifying the stable uniformity of a certain class of rings.

\smallskip

\section{Rational Localizations}

We review the theory of rational localizations of Huber pairs. We spell out the main definitions from \cite{H1}. One reason for doing this is that the construction of (uncompleted) rational localizations does not show up much once the foundational aspects of the theory are developed, but we will really need it in our proof.

\begin{defn}\label{defn:rational-localization} Let $A$ be a Huber ring with a pair of definition $(A_0, I)$ and elements $f_1, \dots, f_n, s \in A$ such that $f_1A +f_2A + \dots +f_n A$ is an open ideal in $A$. 
\begin{itemize}
\item The {\it rational localization} $A\left(\frac{f_1}{s}, \dots, \frac{f_n}{s}\right)$ is a Huber ring such that:
\begin{enumerate}
    \item The ring structure is given by $A\left(\frac{f_1}{s}, \dots, \frac{f_n}{s}\right)=A\left[\frac{1}{s}\right]$.
    \item A ring of definition is given by 
    \[
    A_0\left[\frac{f_1}{s}, \dots, \frac{f_n}{s}\right] \subset A\left[\frac{1}{s}\right],
    \]
    where $A_0\left[\frac{f_1}{s}, \dots, \frac{f_n}{s}\right]$ is the $A_0$-subagebra of $A\left[\frac{1}{s}\right]$ generated by $\frac{f_1}{s}, \dots, \frac{f_n}{s}$ (in particular, $A_0\left[\frac{f_1}{s}, \dots, \frac{f_n}{s}\right]$ depends on $A$ and not only on $A_0$).

    \item An ideal of definition is given by $IA_0\left[\frac{f_1}{s}, \dots, \frac{f_n}{s}\right] \subset A\left[\frac{1}{s}\right]$.
\end{enumerate}
\item The {\it completed rational localization} $A\left\langle\frac{f_1}{s}, \dots, \frac{f_n}{s}\right\rangle$ is defined as the completion of the Huber ring $A\left(\frac{f_1}{s}, \dots, \frac{f_n}{s}\right)$.
\end{itemize}
\end{defn}

\begin{rmk} One can check that $A\left(\frac{f_1}{s}, \dots, \frac{f_n}{s}\right)$ is well-defined, i.e. it is indeed a Huber ring and it is independent of a choice of a couple of definition $(A_0, I)$. See \cite[Lemma and definition on p.516 and the universal property (1.2) on p.517]{H1}.
\end{rmk}

\begin{rmk}\label{rmk:completion-pair} \cite[Lemma 1.6(ii)]{H0} implies that $A\left\langle\frac{f_1}{s}, \dots, \frac{f_n}{s}\right\rangle$ is a Huber ring with a ring of definition equal to 
\[
A_0\left\langle\frac{f_1}{s}, \dots, \frac{f_n}{s}\right\rangle \coloneqq A_0\left[\frac{f_1}{s}, \dots, \frac{f_n}{s}\right]^{\wedge}, 
\]
and an ideal of definition $IA_0\left\langle\frac{f_1}{s}, \dots, \frac{f_n}{s}\right\rangle$. 
\end{rmk}

\begin{rmk}\label{rmk:rational-values} The main importance of this construction is that it gives values of the structure presheaf on rational subdomains. More precisely, suppose that $X=\Spa(A, A^+)$ for a complete Huber pair $(A, A^+)$. Then we have a topological isomorphism
\[
\O_X\left(X(\frac{f_1}{s}, \dots, \frac{f_n}{s})\right) \simeq A\left\langle\frac{f_1}{s}, \dots, \frac{f_n}{s} \right\rangle
\]
for any $f_1, \dots, f_n, s\in A$ such that the ideal $f_1A+ \dots +f_nA$ is open in $A$.  
\end{rmk}

We also review a slightly more general version of this construction that will be convenient for our later purposes.

\begin{defn}\label{defn:rational-localization-may} Let $A$ be a Huber ring, $s_1, s_2, \dots, s_n$ elements of $A$, and finite sets $F_1, F_2, \dots, F_n$ of elements of $A$ such that the ideal generated by $F_i$ is open. Let $(A_0, I)$ be a pair of definition. 
\begin{itemize}
\item The {\it rational localization} $A\left(\frac{F_1}{s_1};\dots; \frac{F_n}{s_n}\right)$ is a Huber ring such that:
\begin{enumerate}
    \item The ring structure is given by $A\left(\frac{F_1}{s_1};\dots; \frac{F_n}{s_n}\right)=A\left[\frac{1}{s_1}, \dots, \frac{1}{s_n}\right]$.
    \item A ring of definition is given by 
    \[
    A_0\left[\frac{F_1}{s_1}; \dots; \frac{F_n}{s_n}\right] \coloneqq A_0\left[\frac{f}{s_i} \ | \ i=1, \dots, n, f\in F_i\right] \subset A\left[\frac{1}{s_1}, \dots, \frac{1}{s_n}\right]. 
    \]
    \item An ideal of definition is given by $IA_0\left[\frac{F_1}{s_1}; \dots; \frac{F_n}{s_n}\right]\subset A\left(\frac{F_1}{s_1};\dots; \frac{F_n}{s_n}\right)$.
\end{enumerate}
\item The {\it completed rational localization} $A\left\langle\frac{F_1}{s_1};\dots; \frac{F_n}{s_n}\right\rangle$ is defined as the completion of the Huber ring $A\left(\frac{F_1}{s_1};\dots; \frac{F_n}{s_n}\right)$.
\end{itemize}
\end{defn}

\begin{rmk} If we set $F=\{f_1,\dots, f_n\}$, it is clear that 
\[
    A\left(\frac{F}{s}\right)=A\left(\frac{f_1}{s}, \dots, \frac{f_n}{s}\right), \     A\left\langle\frac{F}{s}\right\rangle=A\left\langle\frac{f_1}{s}, \dots, \frac{f_n}{s}\right\rangle
\]
\[
A_0\left[\frac{F}{s}\right] = A_0\left[\frac{f_1}{s}, \dots, \frac{f_n}{s}\right], \ A_0\left\langle\frac{F}{s}\right\rangle = A_0\left\langle\frac{f_1}{s}, \dots, \frac{f_n}{s}\right\rangle.
\] 
\end{rmk}

\begin{rmk}\label{rmk:rational-values-many} Similarly to Remark~\ref{rmk:rational-values}, we have a canonical topological isomorphism
\[
\O_X\left(X(\frac{F_1}{s_1}) \cap \dots \cap X(\frac{F_n}{s_n})\right) \simeq A\left\langle \frac{F_1}{s_1}; \dots: \frac{F_n}{s_n} \right\rangle
\]
for any Huber pair $(A, A^+)$ with complete $A$, elements $s_1, \dots, s_n \in A$, and finite sets $F_1, \dots, F_n \subset A$ such that the ideal generated by $F_i$ is open for any $i$.
\end{rmk}

\begin{defn}\label{defn:rigid-noetherian} Let $(A_0, I)$ be a pair of a ring $A_0$ and a finitely generated ideal $I$. We say that $A_0$ is {\it topologically universally rigid-noetherian} if $\Spec \wdh{A_0}\langle X_1, \dots, X_d\rangle$ is noetherian outside $I\wdh{A_0}\langle X_1, \dots, X_d\rangle$ for every $d\geq 0$\footnote{This definition differs from \cite[Definition 0.8.4.3]{FujKato} as we do not require $A_0$ to be noetherian outside $I$.}. \smallskip

A Huber ring $A$ is {\it strongly rigid-noetherian} if $A$ admits a pair of definition $(A_0, I)$ that is topologically universally rigid-noetherian. 
\end{defn}

\begin{rmk} We want to emphasize that strong rigid-noetherianness of $A$ does not imply that $A$ is noetherian. For instance, the ring $\O_{\C_p}$ is strongly rigid-noetherian, but $\O_{\C_p}$ is not noetherian. 
\end{rmk}

\begin{rmk}\label{rmk:strongly-noeth} The definition of a strongly rigid-noetherian Huber pair does not depend on a choice of a pair of definition $(A_0, I)$. Indeed, it clearly does not depend on a choice of an ideal of definition $I$ inside a fixed ring of definition $A_0$. \smallskip

Now we may and do assume that $A$ is complete. Suppose $A_0$ and $A_1$ are two rings of definition, so \cite[Corollary 1.3]{H0} implies that $A_0\cdot A_1$ is again a ring of definition, so it suffices to show that claim under the additional assumption that $A_0 \subset A_1$. If $I$ is an ideal of definition in $A_0$, then $IA_1$ is an ideal of definition in $A_1$. So it is enough to show that $\Spec A_0$ is noetherian outside $f\in I$ if and only if so is $\Spec A_1$\footnote{Then apply the same reasoning to $A_0\langle \ud{T}\rangle$ and $A_1\langle \ud{T}\rangle$.}.  Now \cite[Lemma 3.7]{H0} ensures that $(A_0)_f \to A_f$ and $(A_1)_f \to A_f$ are isomorphisms. This finishes the proof.
\end{rmk}

\begin{lemma}\label{lemma:analytic-sheafy} Let $A$ be a complete analytic ring (in the sense of \cite[Definition 1.1.2]{KedAr}). Then $A$ is strongly noetherian if and only if it is strongly rigid-noetherian. In particular, a complete Tate ring is strongly noetherian if and only if it is strongly rigid-noetherian.
\end{lemma}
\begin{proof}
    Pick a pair of definition $(I, A_0)$ in $A$. We need to check that, for any $n\geq 0$, a scheme $\Spec A_0\langle T_1, \dots, T_n\rangle \setminus \rm{V}(IA_0\langle T_1, \dots, T_n\rangle)$ is noetherian if and only if $A\langle T_1, \dots, T_n\rangle$ is noetherian. After replacing $A$ with $A\langle T_1, \dots, T_n\rangle$, it suffices to show it for $n=0$. Therefore, it is enough to show that the natural morphism
    \[
    \Spec A \to \Spec A_0 \setminus \rm{V}(I)
    \]
    is an isomorphism. Now \cite[Lemma 1.1.3]{KedAr} ensures that the only open ideal in $A$ is trivial, and thus the the claim follows from \cite[Lemma 3.7]{H0}. 
\end{proof}

\begin{lemma}\label{lemma:sheafy-formal} Let $k^+$ be a complete microbial valuation ring (in the sense of Definition~\cite[Definition 9.1.4]{Seminar}), and $A$ topologically finite type $k^+$-algebra. Then $A$ is a strongly rigid-noetherian. 
\end{lemma}
\begin{proof}
    Pick a pseudo-uniformizer $\varpi\in k^+$. We need to show that $A\langle T_1,\dots, T_n\rangle[\frac{1}{\varpi}]$ is noetherian for any $n\geq 0$. A $k^+$-algebra $A\langle T_1,\dots, T_n\rangle$ is topologically finite type, e.g. there is a surjection
    \[
    k^+\langle T_1, \dots, T_m \rangle  \to A\langle T_1, \dots, T_n\rangle.
    \]
    Therefore, it suffices to show that $k^+\langle T_1, \dots, T_m \rangle[\frac{1}{\varpi}]$ is noetherian for any $m\geq 0$. But this is just the usual Tate algebra $k\langle T_1, \dots, T_m \rangle$ over the non-archimedean field $k\coloneqq \rm{Frac}(k^+)$. It is noetherian by \cite[Proposition 2.2/14]{B}.
\end{proof}

\begin{lemma}\label{lemma:strongly-noeth-subspace} Let $A$ be a strongly rigid-noetherian Huber ring, and $f_1, \dots, f_n, s \in A$ elements such that $f_1A +f_2A + \dots +f_n A$ is an open ideal in $A$. Then the completed rational localization \[A\left\langle\frac{f_1}{s}, \dots, \frac{f_n}{s}\right\rangle\] is a strongly rigid-noetherian Huber ring.
\end{lemma}
\begin{proof}
Without loss of generality, we can assume that $A$ is complete, equivalently, any ring of definition $A_0$ is complete. Now, it suffices to show that $A\left\langle\frac{f_1}{s}, \dots, \frac{f_n}{s}\right\rangle$ admits a topologically universally rigid-noetherian pair of definition. Remark~\ref{rmk:completion-pair} implies that this ring admits a pair of definition $\left(A_0\left\langle\frac{f_1}{s}, \dots, \frac{f_n}{s}\right\rangle, IA_0\left\langle\frac{f_1}{s}, \dots, \frac{f_n}{s}\right\rangle\right)$. Clearly, there a surjection
\[
A_0\left\langle X_1, \dots, X_n\right\rangle \to A_0\left\langle\frac{f_1}{s}, \dots, \frac{f_n}{s}\right\rangle.
\]
Thus, $A_0\left\langle\frac{f_1}{s}, \dots, \frac{f_n}{s}\right\rangle$ is a topologically finitely generated $A_0$-algebra. Therefore, the pair \[\left(A_0\left\langle\frac{f_1}{s}, \dots, \frac{f_n}{s}\right\rangle, IA_0\left\langle\frac{f_1}{s}, \dots, \frac{f_n}{s}\right\rangle\right)\] is topologically universally rigid-noetherian as $(A_0, I)$ is so.
\end{proof}

\begin{defn}\label{defn:pseudo-adhesive} A pair $(A, I)$ of a ring $A$ and a finitely generated ideal $I\subset A$ is {\it pseudo-adhesive} (or $A$ is {\it $I$-adically pseudo-adhesive}) if $\Spec A$ is noetherian outside $\rm{V}(I)$ and any finite $A$-module $M$ has bounded $I$-power torsion (i.e. $M[I^{\infty}]=M[I^n]$ for some $n$). \smallskip

A pair $(A, I)$ of finite type is {\it universally pseudo-adhesive} (or $A$ is {\it $I$-adically universally pseudo-adhesive}) if $(A[X_1, \dots, X_d], IA[X_1, \dots, X_d])$ is pseudo-adhesive for any $d\geq 0$.
\end{defn}

\begin{rmk}\label{rmk:adhesive-finite-type} It is easy to see that any finite type $A$-algebra over a universally pseudo-adhesive pair $(A, I)$ is $I$-adically universally pseudo-adhesive.
\end{rmk}

The following theorem of Fujiwara, Gabber, and Kato will play a crucial role in what follows. It will give us a way to apply results from Appendix~\ref{section:FP} in our context.

\begin{thm}\label{thm:loc-Gabber}\cite[Theorem 0.8.4.8]{FujKato} Let $(A_0, I)$ be a complete topologically universally rigid-noetherian pair. Then it is universally pseudo-adhesive. 
\end{thm}

\section{Sheafiness Of Strognly Noetherian Huber Pairs}

We show that any strongly rigid-noetherian Huber ring $A$ is sheafy. Our proof follows Huber's proof of the same result for Huber pairs with a noetherian ring definition very closely (see \cite[Theorem 2.2]{H1}). The main obstacle why his proof does not work in this more general situation is that he needs to use certain finiteness of cohomology groups from \cite{EGA3} that require the noetherian hypothesis. Instead, we use the results from Section~\ref{section:FP} in place of the results from \cite{EGA3}. However, we want to point out one complication is that Theorem~\ref{thm:FP-approx-proper} does not give an honest finiteness result. 

The following lemma plays a crucial role in our argument:

\begin{lemma}\label{lemma:proof-standard-covering}\cite[Lemma 2.6]{H1} Let $(A, A^+)$ be a complete Huber pair, and $\{V_j\}_{j\in J}$ be an open
covering of $X=\Spa(A, A^+)$. Then there exist $f_0, \dots, f_n\in A$ such that $A  = f_0A + f_1A +\dots +f_n A$ and, for every $i \in \{0, \dots , n\}$, the rational subset $ X\left(\frac{f_0}{f_i}, \dots, \frac{f_n}{f_i}\right)$ is contained in some $V_j$. 
\end{lemma}

\begin{defn}\label{defn:proof-standard-covering} A {\it standard covering} of $X=\Spa(A, A^+)$ is a covering of the form 
\[
    X=\bigcup_{i=0}^n X\left(\frac{f_0}{f_i}, \dots, \frac{f_n}{f_i}\right)
\]
for some $f_0, \dots, f_n\in A$ such that $A  = f_0A + f_1A +\dots +f_n A$. 
\end{defn}

\begin{defn} A morphism of topological groups $\varphi \colon A\to B$ is called {\it strict} of it is continuous and $A \to \varphi(A)$ is open for the subspace topology on the target.
\end{defn}



\begin{defn} A presheaf of topological rings $\F$ on a topological space $X$ is a {\it sheaf} if $\F$ is a sheaf of sets and the natural map $\F(U) \to \prod_{i\in I} \F(U_i)$ is a topological embedding for any covering $U=\cup_{i\in I} U_i$.
\end{defn}

\begin{thm}\label{thm:main} Let $(A, A^+)$ be a strongly rigid-noetherian Huber pair. Then the structure presheaf $\O_X$ is a sheaf of topological rings on $X=\Spa(A, A^+)$. Furthermore, $\rm{H}^i(U, \O_X)=0$ for any rational subdomain $U\subset X$ and  $i\geq 1$.
\end{thm}
One can actually prove the same result for sheaves $M\otimes \O_X$ for any finite $A$-module $M$. The proof should be slightly modified as done in \cite[Theorem 2.5]{H1}. We prefer to write the argument only in the case of the structure presheaf as it simplifies the exposition significantly. 
\begin{proof}
{\it Step 0. We may assume that $A$ is complete}: This follows from the fact that there is a canonical isomorphism 
\[
\left(\Spa(A, A^+), \O_{A, A^+}\right) \simeq \left(\Spa(\wdh{A}, \wdh{A}^+), \O_{\wdh{A}, \wdh{A}^+}\right).
\] 

{\it Step 1. We reduce theorem to showing that $\check{C}_{aug}^\bullet(\cal{U}, \O_X)$ is exact with strict differentials for a standard covering $\cal{U}=\{U_0, \dots, U_n\}$ of $X$}: The sheaf condition means that the sequence
\[
0 \to \O_X(U) \xr{d} \prod_i \O_X(U_i) \to \prod_{i<j} \O_{X}(U_i\cap U_j)
\]
is a exact with strict $d$ for any for any covering $\cal{U}$ of an open $U$. Since rational subsets form a basis of $X$, it suffices to show the claim for a covering of a rational subdomain $U\subset X$ by rational subdomains $U_i\subset X$. Then Lemma~\ref{lemma:proof-standard-covering} allows us to assume that the covering is standard. Lemma~\ref{lemma:strongly-noeth-subspace} shows that $U=\Spa(B, B^+)$ is an affinoid with a strongly rigid-noetherian complete Huber ring $B$. So we may replace $X$ by $U$.\smallskip

Likewise, the \v{C}ech-to-derived spectral sequence and Lemma~\ref{lemma:proof-standard-covering} implies that it is sufficient to show that $\check{\rm{H}}^i(\cal{U}, \O_X)=0$ for any $i> 0$, rational $U$, and any standard covering $\cal{U}$ of $U$. Lemma~\ref{lemma:strongly-noeth-subspace} ensures that we can replace $X$ with $U$ to assume that $\cal{U}$ is a covering of $X$. \smallskip

Therefore, we reduced the original question to show that the augmented (alternating) $\check{C}$ech complex
\[
\check{C}_{aug}^\bullet(\cal{U}, \O_X)\coloneqq (A[1] \to \check{C}^\bullet(\cal{U}, \O_X))
\]
is exact with strict differentials for any standard covering $\cal{U}$ of $X$. \medskip

{\it Step 2. We show that the ``decompleted'' augmented $\check{C}$ech complex is exact}: Now suppose that the standard covering is given by elements $f_0, \dots, f_n \in A$ with $f_0A + \dots + f_n A =A$. Then we choose a  pair of definition $(A_0, I)$ with a topologically universally rigid-noetherian $A_0$. We consider the $A_0$-module $J\coloneqq f_0A_0 + f_1A_0 + \dots + f_nA_0$ inside $A$. We denote $S\coloneqq \Spec A_0$, $U\coloneqq \Spec A$, $P\coloneqq \mathrm{Proj}\oplus J^m\footnote{The notation $J^m$ means the $A_0$-submodule of $A$ generated by all $m$-fold products of elements in $J$. In particular, $J^0=A_0$.}$, and $P'\coloneqq \mathrm{Proj}\oplus (JA)^m$. Then we have a commutative square
\[
\begin{tikzcd}
P' \arrow{r}{p} \arrow{d}{s} & U \arrow{d}{j} \\
P \arrow{r}{g} & S.
\end{tikzcd}
\]
Clearly, $p$ is an isomorphism as $JA=A$, so $s$ induces a morphism $s\colon U \to P$ such that the diagram
\[
\begin{tikzcd}
& U \arrow[dl, swap, "s"] \arrow{d}{j} \\
P \arrow{r}{g} & S
\end{tikzcd}
\]
is commutative. We note that $s$ is an affine morphism as $j=g\circ s$ is affine and $g$ is separated. Therefore,  $\rm{R}^is_*\O_U$ vanish for $i>0$. This implies that $\rm{H}^i(P, s_*\O_U)=\rm{H}^i(U, \O_U)=0$ for $i>0$, and $\rm{H}^0(P, s_*\O_U)=A$. \smallskip

Now we compute the same cohomology groups in a different way using the \v{C}ech complex. We choose an affine covering $\cal{P}\coloneqq \{\rm{D}_+(f_i)\}$ of $P$. Since $s$ is quasi-compact and quasi-separated, the $\O_P$-module $s_*\O_U$ is quasi-coherent. So we can compute its cohomology via the \v{C}ech complex. Consider 
\[
C^\bullet\coloneqq \check{C}^\bullet(\cal{P}, s_*\O_U).
\]
The above computation of the cohomology groups $\rm{H}^i(P, s_*\O_U)$ implies that the augmented \v{C}ech complex
\[
C^\bullet_{aug}=\left( A[1] \to C^\bullet\right)
\]
is exact. For brevity, write $F=\{f_0, \dots, f_n\}$. Now we note that 
\begin{align*}
    C^i_{aug}&=\prod_{j_0< j_1< \dots< j_i} s_*\O_U\left(\rm{D}_+\left(f_{j_0}\right)\cap \dots \cap  \rm{D}_+\left(f_{j_i}\right)\right) \\
    &\simeq \prod_{j_0< j_1< \dots< j_i} \O_U\left(\rm{D}(f_{j_0})\cap \dots \cap  \rm{D}\left(f_{j_i}\right)\right) \\
    &\simeq \prod_{j_0< j_1< \dots< j_i} A\left[\frac{1}{f_{j_0}\dots f_{j_i}}\right] \\
    & \simeq \prod_{j_0< j_1< \dots< j_i} A\left(\frac{F}{f_{j_0}}; \frac{F}{f_{j_1}};\cdots ; \frac{F}{f_{j_i}}\right).
\end{align*}
We topologize it as in Definition~\ref{defn:rational-localization-may}. Therefore, we see that after completing it with respect to this topology, we get
\[
\wdh{C}^i_{aug}= \prod_{j_0< j_1< \dots< j_i} A\left\langle\frac{F}{f_{j_0}}; \cdots; \frac{F}{f_{j_i}}\right\rangle = \check{C}^i_{aug}(\cal{U}, \O_X). 
\]
In other words, we see that $C^i_{aug}$ is a ``decompletion'' of $\check{C}^i_{aug}(\cal{U}, \O_X)$. Now we invoke \cite[III.2.12, Lemma 2]{Bou} that says that $\check{C}^i_{aug}(\cal{U}, \O_X)\simeq \wdh{C}^\bullet_{aug}$ is exact with strict differentials if $C^\bullet_{aug}$ is so. We already know that it is exact, so we are left to show that the differentials of $C^\bullet_{aug}$ are strict. \medskip 

{\it Step 3. We reduce the claim to showing that the differentials $d^i\colon \check{C}^i_{aug}(\cal{P}, \O_P) \to \ker d^{i+1}$ are open}: 
We start by considering the natural morphism $\O_P \to s_*\O_U$. We compute this map on the affine opens $\rm{D}_+(f_{j_0}) \cap \dots \cap \rm{D}_+(f_{j_i}) = \rm{D}_+(f_{j_0}\dots f_{j_i})$. We note that, for $F=\{f_{0}, \dots, f_{n}\}$, 
\[
\rm{D}_+(f_{j_0}\dots f_{j_i})=\Spec \left(\oplus J^m\right)_{(f_{j_0}\dots f_{j_i})}\simeq \Spec A_0\left[\frac{f}{f_{j_k}} \ | \ k=0, \dots, i, f\in F\right]  \simeq \Spec A_0\left[\frac{F}{f_{j_0}}; \cdots; \frac{F}{f_{j_i}}\right]. 
\]
We topologize it using the $I$-adic topology. Then the map 
\[
    \O_P\left(\rm{D}_+\left(f_{j_0}\right) \cap \dots \cap \rm{D}_+\left(f_{j_i}\right)  \right) \to s_*\O_U\left(\rm{D}_+\left(f_{j_0}\right) \cap \dots \cap \rm{D}_+\left(f_{j_i}\right)\right)=\O_U\left(\rm{D}\left(f_{j_0}\right) \cap \dots \cap \rm{D}\left(f_{j_i}\right)\right)
\]
is naturally identified with
\[
A_0\left[\frac{F}{f_{j_0}}; \cdots; \frac{F}{f_{j_i}}\right] \to A\left(\frac{F}{f_{j_0}}; \cdots; \frac{F}{f_{j_i}}\right).
\]
In particular, it is an injective, continuous and {\it open} morphism since $A_0\left[\frac{F}{f_{j_0}}; \cdots; \frac{F}{f_{j_i}}\right]$ is a ring of definition in $A\left(\frac{F}{f_{j_0}}; \cdots; \frac{F}{f_{j_i}}\right)$. Therefore, for any $i\geq 0$, we conclude that
\[
\check{C}^i_{aug}(\cal{P}, \O_P) \to C^i_{aug}
\]
is injective and identifies $\check{C}^i_{aug}(\cal{P}, \O_P)$ with a ring of definition of $C^i_{aug}$. \smallskip

Now we deal with the case of $i=-1$, separately. We note that 
\[
    \check{C}^{-1}_{aug}(\cal{P}, \O_P) \simeq \rm{H}^0(P, \O_P) \subset \check{C}^0_{aug}(\cal{P}, \O_P)
\]
and 
\[
C^{-1}_{aug}\simeq \rm{H}^0(U, \O_U) \subset C^0_{aug}.
\]
Therefore, injectivity of $\check{C}^{-1}_{aug}(\cal{P}, \O_P) \to C^{-1}_{aug}$ follows from injectivity in degree $0$. So we only need to topologize $\check{C}^{-1}_{aug}(\cal{P}, \O_P)$ in a way that $\check{C}^{-1}_{aug}(\cal{P}, \O_P)$ is a ring of definition of $C^{-1}_{aug}=A$. \smallskip

We topologize it using the subspace topology from $\check{C}^0(\cal{P}, \O_P)$. This topology coincides with the natural $I$-topology (see Definition~\ref{defn:FP-approx-natural-topology}) by Remark~\ref{rmk:natural-topology-cech-complex}. Now Theorem~\ref{thm:FP-approx-topology-cohomology}\footnote{Note that $A_0$ is universally pseudo-adhesive by Theorem~\ref{thm:loc-Gabber}} ensures that this topology is the $I$-adic topology. So we need to show that 
\[
    B\coloneqq \rm{H}^0(P,\O_P)\subset A
\]
is an open subring and the subspace topology is the $I$-adic topology.\smallskip

Now we note that the morphism $g^{*}\O_S \to \O_P$ gives the morphism $A_0 \to B$ such that when composed with the inclusion $B \to A$ it is equal to the embedding $A_0 \to A$. This implies that $I^m \subset I^mB$, so it suffices to show that, for any $k$, there is an $m$ such that $I^mB \subset I^k$. \smallskip

Theorem~\ref{thm:FP-approx-proper} guarantees that $B$ is FP-approximated as an $A_0$-module, i.e. there is a finite $A_0$-submodule $M\subset B$ such that the module quotient is annihilated by $I^d$ for some $d$. Since $I$ is an ideal of definition in $A$, and $M$ is finitely generated, we can find $c$ such that $I^cM \subset I^k$. Therefore, $I^{c+d}B \subset I^cM \subset I^k$. This finishes the argument. \smallskip

Overall, we see that the $\check{C}^i_{aug}(\cal{P}, \O_P) \to C^i_{aug}$ is  injective and identifies $\check{C}^i_{aug}(\cal{P}, \O_P)$ with a ring of definition in $C^i_{aug}$ for every $i\geq -1$. \smallskip

Now, it suffices to show that the differentials $d^i_C \colon C^i_{aug} \to \ker d^{i+1}_C$ are open to conclude that $d^i_C\colon \check{C}^i_{aug} \to \check{C}^{i+1}_{aug}$ are strict for every $i\geq -1$. We claim that it is actually sufficient to show that the differentials $\delta^i \colon K^i \to \ker \delta^{i+1}$ is open, where $K^\bullet\coloneqq \check{C}^i_{aug}(\cal{P}, \O_P)$ and $\delta$ is the differential of this complex. \smallskip 

Grant this opennness. We just need to deduce that $d^i_C(I^mK^i)$ is open for any $m\geq 0$ as $\{I^mK^i\}$ for a fundamental system of neighborhoods of $0$ in $C^i_{aug}$ ($K^i$ is a ring of definition in $C^i_{aug}$). We know that 
\[
d^i_C(I^mK^i)=\delta^i(I^mK^i)
\]
is open in $\ker \delta^{i+1}=\ker d_C^{i+1}\cap K^{i+1}$. So as $K^{i+1}$ is open in $C_{aug}^{i+1}$, we conclude that $\ker \delta^{i+1}$ is open in $\ker d_C^{i+1}$. As a result, we get that $d^i_C(I^mK^i)$ is open in $C^{i+1}_{aug}$ for every $m\geq 0$, $i\geq -1$. \smallskip

{\it Step 4. We show that the differentials of $\delta^i \colon K^i \to \ker \delta^{i+1}$ are open}: The claim for $\delta^i$ is trivial if $i<-1$. If $i=-1$, the map $\delta^{-1}\colon K^{-1} \to \ker \delta^0$ is even a homeomorphism because 
\[
K^{-1}=\ker \delta^0
\]
and topology on $K^{-1}$ was defined to be the subspace topology. \smallskip

We consider the restriction $\delta^i \colon K^i \to \ker \delta^{i+1}$, where the target is endowed with the subspace topology. This map is open if and only if, for each $k$, there is $m$ such that 
\[
\ker \delta^{i+1} \cap I^m K^{i+1} \subset I^k \delta^i(K^i)=\delta^i(I^kK^i)
\]
Now we note that 
\[
I^m K^\bullet \simeq I^m \check{C}^\bullet_{aug}(\cal{P}, \O_P) \simeq  \check{C}^\bullet_{aug}(\cal{P}, I^m\O_P) =: (K_m^\bullet, \delta_m).
\]  
Then it suffices to show that, for any $k$, there is $m$ such that 
\[
\ker \delta^{i+1}\cap K_m^{i+1} \subset \delta^i_k(K_k^i) 
\]
that is equivalent to 
\[
\ker \delta_m^{i+1} \subset \delta^i_k(K_k^i)
\]
This means that we need to find $m$ such that 
\[
\rm{H}^{i+1}(K_m) \to \rm{H}^{i+1}(K_k)
\]
is zero. Unravelling the definitions, we get that this is equivalent to find $m$ such that
\[
\rm{H}^{i+1}(P, I^m\O_P) \to \rm{H}^{i+1}(P, I^k\O_P)
\]
is zero. \smallskip

Now we prove that claim under the assumption that $I^c\rm{H}^{i+1}(P, I^k\O_P)=0$ for some $c$ (depending on $k$ and $i\geq 0$) and then we show that this assumption always holds. We firstly observe that 
\[
\rm{Im}\left(H^{i+1}(P, I^m\O_P) \to H^{i+1}(P, I^k\O_P)\right)=\rm{F}^{m-k}\rm{H}^{i+1}(P, I^k\O_P),
\]
where $\rm{F}^\bullet$ stands for the natural $I$-filtration (see Definition~\ref{defn:FP-approx-natural-topology}). Now we note that $I^k\O_P$ is a finitely generated, quasi-coherent $\O_P$-module, so it is FP-approximated by Lemma~\ref{lemma:FP-approx-finite-type}. Therefore, Theorem~\ref{thm:FP-approx-topology-cohomology} ensures that the natural $I$-topology on $\rm{H}^{i+1}(P, I^k\O_P)$ is the $I$-adic topology. Then there is some $d$ such that
\[
\rm{F}^d\rm{H}^{i+1}\left(P, I^k\O_P\right) \subset I^c\rm{H}^{i+1}\left(P, I^k\O_P\right)=0.
\]
Claim~\ref{claim:killed-by-I^n} below ensures that $\rm{H}^{i+1}(P, I^k\O_P)$ is indeed annihilated by some $I^c$ for some $c$ depending on $k$ and $i\geq 0$.

\begin{claim1}\label{claim:iso-away-from-I} The morphism $g\colon P \to S$ is an isomorphism away from $\rm{V}(I)$.
\end{claim1}
\begin{proof}
It suffices to show that $g$ is isomorphism over $\rm{D}(f)$ for any $f\in I$. We note that 
\[
\left(\mathrm{Proj} \bigoplus J^m \right)\times_{\Spec A_0} \Spec (A_0)_f \simeq \mathrm{Proj} \bigoplus \left(J\left(A_0\right)_f\right)^m
\]
as $A_0 \to (A_0)_f$ is flat. Therefore, it suffices to show that $(A_0)_f\simeq A_f$ as then 
\[
J(A_0)_f=JA_f=(JA)A_f=A_f=(A_0)_f,
\] 
and so $g$ is an isomorphism over $(A_0)_f$. Now (the proof of) \cite[Lemma 3.7]{H0} implies that the natural map $(A_0)_f \to A_f$ is an isomorphism. 
\end{proof}

\begin{claim1}\label{claim:killed-by-I^n} For any $i, k\geq 0$, there is $c$ such that $I^c\rm{H}^{i+1}(P, I^k\O_P)=0$.
\end{claim1}
\begin{proof}
We note that $g$ is quasi-compact and separated, so \[
\rm{R}^{i+1}g_*\left(I^k\O_P\right) \simeq \widetilde{\rm{H}^{i+1}(P, I^k\O_P)}.
\]
Now Claim~\ref{claim:iso-away-from-I} says that $g$ is an isomorphism over $\Spec A_0 \setminus \rm{V}(I)$, so since $i\geq 0$ we have
\[
\rm{R}^{i+1}g_*\left(I^n\O_P\right)|_{\Spec A_0 \setminus \rm{V}(I)} \simeq 0.
\]
Since $I$ is finitely generated, this says 
\[
\rm{H}^{i+1}(P, I^k\O_P)=\rm{H}^{i+1}(P, I^k\O_P)[I^{\infty}]. 
\]
As $g$ is projective, Theorem~\ref{thm:FP-approx-proper} implies that $\rm{H}^{i+1}(P, I^k\O_P)$ is an FP-approximated $A_0$-module. Therefore, Lemma~\ref{lemma:FP-approx-bounded-torsion} ensures that for some $c\geq 0$ we have
\[
\rm{H}^{i+1}(P, I^k\O_P)=\rm{H}^{i+1}(P, I^k\O_P)[I^{\infty}]=\rm{H}^{i+1}(P, I^k\O_P)[I^c];
\]
i.e. $\rm{H}^{i+1}(P, I^k\O_P)$ is annihilated by $I^c$.
\end{proof}
\end{proof}

\newpage

\appendix

\section*{Appendix}

\section{FP-Approximated Sheaves}\label{section:FP}

This section is a summary of the results from \cite[Appendix C to Chapter I]{FujKato}. However, some of them were only announced in that Appendix, but no proof was given. Since these results are crucial for our proof of Theorem~\ref{thm:main}, we decided to provide the reader with the proofs in the generality we need in this paper. All main ideas are already present in \cite{FujKato} \smallskip

For the rest of the appendix, we fix a universally pseudo-adhesive pair $(R, I)$ (see Definition~\ref{defn:pseudo-adhesive}). In particular, $\Spec R$ noetherian outside $\rm{V}(I)$ and $I$ is finitely generated. \smallskip

We recall that an $R$-scheme is {\it universally $I$-adically pseudo-adhesive} (or simply {\it universally pseudo-adhesive}) if it has a covering by open affines $\Spec A_i$ such that each $A_i$ is $I$-adically universally pseudo-adhesive (see \cite[\textsection \ 0.8.6]{FujKato} for a more detailed discussion of this notion). Any finite type $R$-scheme is  universally pseudo-adhesive by Remark~\ref{rmk:adhesive-finite-type}. In particular, any  quasi-coherent $\O_X$-module of finite type $\F$ has bounded $I$-power torsion, i.e. $\F[I^{\infty}]=\F[I^n]$ for some $n$. \smallskip

Let us mention that the main reason to bring in the pseudo-adhesive assumption is to rescue noetherian techniques for non-noetherian situations with suitable finitely generated ideals. For instance, we will need to ensure that a submodule of a finite $A$-module has some precise finiteness property (Lemma~\ref{lemma:FP-approx-finite-type}) and its subspace topology coincides with the $I$-adic topology (Lemma~\ref{cor:FP-approx-subsheaf-top}).

\begin{defn}\label{defn:FP-approx} 
\begin{enumerate}
    \item A morphism of $\O_X$-modules $\varphi \colon \F \to \G$ is a {\it weak isomorphism} if $\rm{coker} \varphi$ and $\ker \varphi$ are annihilated by $I^n$ for some $n$.
    \item An {\it FP-approximation} of a quasi-coherent $\O_X$-module $\F$ is a weak isomorphism $\varphi \colon \G \to \F$ from a finitely presented $\O_X$-module $\G$. 
    \item An {\it FP-thickening} of a quasi-coherent $\O_X$-module $\F$ is a surjective FP-approximation $\varphi\colon \G \to \F$.
    \item A quasi-coherent $\O_X$-module is {\it FP-approximated} if there is an FP-approximation $\varphi \colon \G \to \F$. 
    \item An $R$-module $M$ is {\it FP-approximated} if $\widetilde{M}$ is an FP-approximated sheaf on $\Spec R$.
    \end{enumerate}
\end{defn}

\begin{lemma}\label{lemma:FP-approx-bounded-torsion} Let $M$ be an FP-approximated $R$-module. Then its $I^{\infty}$-torsion is bounded, i.e. $M[I^{\infty}]=M[I^n]$ for some $n\geq 0$. 
\end{lemma}
\begin{proof}
The definition of FP-approximated modules implies that there is a finite type $R$-submodule $N\subset M$ such that $M/N$ is killed by $I^m$ for some $m$. So we may and do assume that $M$ is an $R$-finite module. This case follows from the definition of pseudo-adhesive pairs. 
\end{proof}

\begin{lemma}\label{lemma:FP-approx-finite-type} Let $X$ be a finite type $R$-scheme. Then 
\begin{enumerate}
    \item\label{lemma:FP-approx-finite-type-1} any quasi-coherent $\O_X$-module of finite type admits an FP-thickening, 
    \item\label{lemma:FP-approx-finite-type-2} the category of FP-approximated sheaves is a Weak Serre abelian subcategory of the category of $\O_X$-modules,
    \item\label{lemma:FP-approx-finite-type-3} Any quasi-coherent sub or quotient sheaf of an FP-approximated $\F$ is FP-approximated.
\end{enumerate}
\end{lemma}
\begin{proof}
 Part~(\ref{lemma:FP-approx-finite-type-1}) is \cite[Proposition I.C.2.2]{FujKato}. Part~(\ref{lemma:FP-approx-finite-type-2}) is \cite[Theorem I.C.2.5]{FujKato}.  \smallskip
 
 We firstly prove Part~(\ref{lemma:FP-approx-finite-type-3}) for  quasi-coherent quotients of $\F$. The definition of FP-approximated sheaves easily implies that there is finite type quasi-coherent $\O_X$-submodule $\G \subset \F$ such that $\F/\G$ is annihilated by some $I^n$. Then if $\pi \colon \F \to \F'$ is a surjective map of quasi-coherent $\O_X$-modules, we define $\G' \coloneqq \pi(\G)$. Clearly, $\G'$ is a quasi-coherent $\O_X$-module of finite type, and $\F'/\G'$ is annihilated by $I^n$. Therefore, Part~(\ref{lemma:FP-approx-finite-type-1}) implies that $\F'$ is FP-approximated. \smallskip
 
 Now if $\F'$ is a quasi-coherent subsheaf of $\F$, it is clear that $\F''\coloneqq \F/\F'$ is quasi-coherent. So $\F''$ is FP-approximated by the discussion above. Thus, Part~(\ref{lemma:FP-approx-finite-type-2}) implies that $\F'$ is FP-approximated because FP-approximated sheaves are closed under kernels. 
\end{proof}

\begin{cor}\label{cor:FP-approx-preserved} Let $i\colon X \to Y$ be a closed immersion of finite type $R$-schemes, and let $\F$ be an FP-approximated $\O_X$-module. Then $i_*\F$ is an FP-approximated $\O_Y$-module.
\end{cor}
\begin{proof}
As $i_*$ is exact, it suffices to show that the corollary holds for finitely presented $\O_X$-modules. So we may and do assume that $\F$ is finitely presented. Then $i_*\F$ is clearly a quasi-coherent $\O_X$-module of finite type. Therefore, it is FP-approximated by Lemma~\ref{lemma:FP-approx-finite-type}(\ref{lemma:FP-approx-finite-type-1}).
\end{proof}

Now we want to study cohomology groups of FP-approximated sheaves on projective $R$-schemes. We show that these cohomology are always FP-approximated $R$-modules, and a certain natural topology on these modules coincides with the $I$-adic topology. These results were announced in the proper case in \cite[Appendix C to Chapter I]{FujKato}. We do not discuss this generalization as the projective case is sufficient for our purposes. 

\begin{thm}\label{thm:FP-approx-proper} Let $X$ be a projective $R$-scheme, and let $\F$ be an FP-approximated $\O_X$-module. Then $\rm{H}^i(X, \F)$ is an FP-approximated $R$-module for any $i\geq 0$.
\end{thm}
\begin{rmk} We do not impose the finite presentation assumption on $X$. The finite presentation version of Theorem~\ref{thm:FP-approx-proper} will be inadequate for the purpose of proving Theorem~\ref{thm:main}. 
\end{rmk}
\begin{proof}
 We firstly reduce to the case $X=\P_{R}^n$. Namely, there is a closed immersion $i \colon X \to \P^n_R$ as $X$ is projective. Since $i_*$ is exact, it suffices to show the claim for the sheaf $i_*\F$ that is FP-approximated by Corollary~\ref{cor:FP-approx-preserved}.  \smallskip
 
 Now we argue that $\rm{H}^i(\P^n_R, \F)$ is an FP-approximated $R$-module by descending induction on $i$. \smallskip
 
We claim that $\rm{H}^i(\P^n_R, \F)=0$ if $i>n$. Indeed, $\P^n_R$ admits the standard affine covering $\cal{U}=\{U_i\}$ by $n+1$ opens. So the cohomology groups of any quasi-coherent sheaf can be computed by the {\it alternating} Cech complex with respect to that covering. Thus, $\rm{H}^i(\P^n_R, \F)=0$ for any $i>n$.\smallskip

Now we do the induction step. Suppose we know the claim for all FP-approximated sheaves $\F$ and all $i>k$, we conclude the statement for $i=k$. By definition, we can find a weak isomorphism $\G \to \F$ with a finitely presented $\O_X$-module $\G$. It is clear that morphisms $\rm{H}^i(\P^n_R, \G) \to \rm{H}^i(\P^n_R, \F)$ are weak isomorphisms for any $i$. Thus it suffices to prove the claim for a finitely presented $\O_X$-module $\F$. \smallskip

We invoke the ample line bundle $\O_{\P^n_R}(1)$ to say that there is always a short exact sequence
\[
0 \to \F' \to \O_{\P^n_R}(r)^m \to \F \to 0
\]
for some {\it negative} $r$. Lemma~\ref{lemma:FP-approx-finite-type}(\ref{lemma:FP-approx-finite-type-2}) implies that $\F'$ is FP-approximated, so $\rm{H}^i(\P^n_R, \F')$ are FP-approximated for any $i>k$ by the induction assumption. \smallskip

Firstly we consider the case $k=n$. Then we know that $\rm{H}^{n+1}(\P^n_R, \F')=0$ by the discussion above. So the natural morphism $\rm{H}^k(\P^n_R, \O_{\P^n_R}(r))^m \to \rm{H}^k(\P^n_R, \F)$ is surjective. This implies that $\rm{H}^k(\P^n_R, \F)$ is a finite $R$-module by Serre's computation. Therefore, it is FP-approximated by Lemma~\ref{lemma:FP-approx-finite-type}(\ref{lemma:FP-approx-finite-type-1}).  \smallskip

Now suppose that $k<n$. Then we know that $\rm{H}^k(\P^n_R, \O_{\P^n_R}(r))^m=0$ by Serre's computation\footnote{We use here that $r<0$.}. Therefore, we conclude that the natural map $\rm{H}^k(\P^n_R, \F) \to \rm{H}^{k+1}(\P^n_R, \F')$ is injective. Thus, $\rm{H}^k(\P^n_R, \F)$ is FP-approximated by Lemma~\ref{lemma:FP-approx-finite-type}(\ref{lemma:FP-approx-finite-type-3}) and the induction assumption.
\end{proof}

Now we try to understand a topology on $\rm{H}^i(X, \F)$. 

\begin{defn}\label{defn:FP-approx-natural-topology}
The {\it natural $I$-filtration} $\rm{F}^\bullet \rm{H}^i(X, \F)$ is
\[
\rm{F}^n \rm{H}^i(X, \F) \coloneqq \rm{Im}\left(\rm{H}^i(X, I^n\F) \to \rm{H}^i(X, \F) \right)
\]

The {\it natural $I$-topology} on $\rm{H}^i(X, \F)$ is the topology induced by the filtration $\rm{F}^\bullet \rm{H}^i(X, \F)$. 
\end{defn}

\begin{rmk}\label{rmk:natural-topology-cech-complex} Suppose $X$ is a separated quasi-compact $R$-scheme, $\F$ a quasi-coherent $\O_X$-module, and $\cal{U}=\{U_1, \dots, U_n\}$ an open affine covering of $X$. Then the natural $I$-topology on $\rm{H}^i(X, \F)$ coincides with the subquotient topology on $\rm{H}^i(X, \F)\simeq \check{\rm{H}}^i(\cal{U}, \F)$ induced from the $I$-adic topology on the (alternating) \v{C}ech complex $\check{C}^i(\cal{U}, \F)$.
\end{rmk}

Clearly $I^n \rm{H}^i(X, \F) \subset \rm{F}^n \rm{H}^i(X, \F)$ for any $n$. These two filtrations on $\rm{H}^i(X, \F)$ are usually different, but we claim that the induced topologies are the same for any FP-approximated sheaf $\F$ on a projective $R$-scheme $X$. \smallskip

Before proving this claim, we need the following lemma:

\begin{lemma}\label{lemma:fp-approx-subtopolgy} Let $M$ be an FP-approximated $R$-module, and $N\subset M$ be any submodule. The the $I$-adic topology on $M$ restricts to the $I$-adic topology on $N$.
\end{lemma}
\begin{proof}
If $M$ is a finitely generated this is proven in \cite[Proposition 0.8.5.6]{FujKato}. \smallskip

Now we deal with the case of any FP-approximated $R$-module $M$. Clearly, $I^nN \subset I^nM\cap N$ for any $n$. So it suffices to show that, for any $n$, there is $m$ such that $I^mM \cap N \subset I^n N$. \smallskip

We can find a finite $R$-submodule $M' \subset M$ such that $M/M'$ is annihilated by $I^c$. Then we know that the $I$-adic topology on $M'$ restricts to the $I$-adic topology on $N'\coloneqq M' \cap N$ by the case of finite $R$-modules. This means that there is an integer $p$ such that $I^pM' \cap N \subset I^n N'$. Then 
\[
I^{c+p} M \cap N \subset I^pM' \cap N \subset I^nN' \subset I^n N.
\]
So $m=c+p$ does the job.
\end{proof}

\begin{cor}\label{cor:FP-approx-subsheaf-top} Let $X$ be a finite type $R$-scheme, $\F$ an FP-approximated sheaf, $\G \subset \F$ be a quasi-coherent $\O_X$-submodule of $\F$. Then, for any $n$, there is $m$ such that $I^m\F\cap \G\subset I^n\G$. 
\end{cor}
\begin{proof}
It suffices to assume that $X$ is affine, in which case it follows from Lemma~\ref{lemma:fp-approx-subtopolgy}. 
\end{proof}

\begin{cor}\label{cor:different-top} Let $X$ be a finite type $R$-scheme, $\G$ an FP-approximated sheaf, and $\varphi\colon \G \to \F$ a weak isomorphism of quasi-coherent $\O_X$-modules. Then, for every $i\geq 0$, the natural $I$-topology on $\rm{H}^i(X, \F)$ coincides with the topology induced by the filtration \[\rm{Fil}_\G^n\rm{H}^i(X, \F)=\rm{Im}(\rm{H}^i(X, I^n\G) \to \rm{H}^i(X, \F)).\]
\end{cor}
\begin{proof}
Consider the short exact sequences
\[
0 \to \cal{K} \to \G \to \cal{H} \to 0,
\]
\[
0 \to \cal{H} \to \F \to \cal{Q} \to 0,
\]
where $\cal{K}$ and $\cal{Q}$ are annihilated by $I^n$ for some $n$. The first short exact sequence induced the short exact sequence
\[
0 \to \cal{K}\cap I^m\G \to I^m\G \to I^m\cal{H} \to 0
\]
for any $m\geq 0$. Corollary~\ref{cor:FP-approx-subsheaf-top} implies that $\cal{K}\cap I^m\G \subset I^n\cal{K}=0$ for large enough $m$. Therefore, the natural map $I^m\G \to I^m\cal{H}$ is an isomorphism for large enough $m$. So we can replace $\G$ with $\cal{H}$ to assume that $\varphi$ is injective (since $\cal{H}$ is FP-approximated by Lemma~\ref{lemma:FP-approx-finite-type}(\ref{lemma:FP-approx-finite-type-3})). \smallskip

Now clearly $\rm{Fil}_\G^k\rm{H}^i(X, \F) \subset \rm{F}^k\rm{H}^i(X, \F)$ for every $k$. So it suffices to show that, for any $k$, there $m$ such that $\rm{F}^m\rm{H}^i(X, \F) \subset \rm{Fil}_\G^k\rm{H}^i(X, \F)$. We consider the short exact sequence
\[
0 \to \G \cap I^m\F \to I^m\F \to I^m\cal{Q} \to 0.
\]
If $m\geq n$ we get that $\G \cap I^m \F = I^m\F$ because $I^m\cal{Q}\simeq 0$. Now we use Corollary~\ref{cor:FP-approx-subsheaf-top} to conclude there is $m\geq n$ such that 
\[
I^m\F = \G\cap I^m\F \subset I^k\G
\]
Therefore, $\rm{F}^m\rm{H}^i(X, \F) \subset \rm{Fil}_\G^k\rm{H}^i(X, \F)$. 
\end{proof}

\begin{lemma}\label{lemma:FP-approx-enough} Let $X$ be a finite type $R$-scheme, $\G$ an FP-approximated sheaf, and $\G \to \F$ a weak isomorphism of quasi-coherent $\O_X$-modules. Suppose that the natural $I$-topology on $\rm{H}^i(X, \G)$ is the $I$-adic topology. Then the same holds for $\rm{H}^i(X, \F)$. 
\end{lemma}
\begin{proof}
Clearly, $I^n\rm{H}^i(X, \F) \subset \rm{F}^n\rm{H}^i(X, \F)$. So it suffices to show that, for every $n$, there is an $m$ such that $\rm{F}^m \rm{H}^i(X, \F) \subset I^n\rm{H}^i(X, \F)$. \smallskip

The assumption that the natural $I$-topology on $\rm{H}^i(X, \G)$ coincides with the $I$-adic topology guarantees that $\rm{F}^k\rm{H}^i(X, \G) \subset I^n\rm{H}^i(X, \G)$ for large enough $k$. Pick such $k$. Corollary~\ref{cor:different-top} implies that 
\[
\rm{F}^m\rm{H}^i(X, \F) \subset \rm{Im}(\rm{H}^i(X, I^k\G) \to \rm{H}^i(X, \F))
\]
for large enough $m$. So we get, for such $m$,  that 
\[
\rm{F}^m\rm{H}^i(X, \F) \subset \rm{Im}\left(\rm{H}^i(X, I^k\G) \to \rm{H}^i\left(X, \F\right)\right) \subset \rm{Im}\left(I^n\rm{H}^i\left(X, \G\right) \to \rm{H}^i\left(X, \F\right)\right) \subset I^n\rm{H}^i\left(X, \F\right)
\]
for a large enough $m$.
\end{proof}

\begin{thm}\label{thm:FP-approx-topology-cohomology} Let $X$ be a projective $R$-scheme, and $\F$ be an FP-approximated $\O_X$-module. Then the natural $I$-topology on $\rm{H}^i(X, \F)$ coincides with the $I$-adic topology for any $i$.
\end{thm}

The proof follows the idea of the proof of the Formal Function Theorem in rigid geometry. Namely, we give a relatively simple argument in the case $I$ is generated by one element, and then argue by induction on the number of generators. See \cite[Proposition 6.4/8]{B} for an example of a classical argument of this form. However, it would be nice to give a proof of Theorem~\ref{thm:FP-approx-topology-cohomology} as a formal consequence of Theorem~\ref{thm:FP-approx-proper} similar to what happens in \cite[Proposition I.8.5.2]{FujKato}.

\begin{proof}
{\it Step 1. Case of a principal ideal $I$}: Suppose that $I$ is a generated by one element $a$. Choose a finite open affine covering $X=\cup_{i=1}^n U_i$ that we denote by $\cal{U}$. Then we define 
\[
C^\bullet\coloneqq \check{C}^\bullet(\cal{U}, \F)
\]
to be the (alternating) Cech complex of $\F$ with respect to the covering $\cal{U}$.  We note that $I^nC^\bullet =\check{C}^\bullet(\cal{U}, I^n\F)$. So we conclude that
\[
\rm{F}^n\rm{H}^i(X, \F)=\rm{Im}(\rm{H}^i(I^nC^\bullet) \to \rm{H}^i(C^\bullet)).
\]
Since the natural $I$-topology on $\rm{H}^i(X, \F)$ is induced from the subspace topology on $\ker d^i$, it suffices to show subspace topology on $\ker d^i \subset C^i$ coincides with the $I$-adic topology. \cite[Lemma 0.8.2.14]{FujKato} ensures that it suffices to verify that $C^i/\ker d^i$ has bounded $a^\infty$-torsion. Since $C^i/\ker d^i$ is naturally a submodule of $C^{i+1}$, it suffices to justify the claim for $C^{i+1}$. Now we recall that $C^{i+1}=\check{C}^{i+1}(\cal{U}, \F)$, so it suffices to show that $\F(U_{j_0}\cap \dots U_{j_{i+1}})$ has bounded $a^\infty$-torsion for all possible $j_0, \dots, j_{i+1}\in [1, n]$. This follows from affinness of each intersection $U_{j_0}\cap \dots \cap U_{j_{i+1}}$ and Lemma~\ref{lemma:FP-approx-bounded-torsion} since $\F$ is FP-approximated. \smallskip

{\it Step 2. The General Case}: We argue by induction on the number of generators $I=(a_1, \dots, a_r)$ over all such $\F$. The claim for $r=1$ was proven in Step~$1$. So we assume that the claim is known for any $i<r$ and all such $\F$, we show that this implies the claim for $r$. \smallskip

Clearly, $I^n\rm{H}^i(X, \F) \subset \rm{F}^n\rm{H}^i(X, \F)$, so it suffices to show that, for any $n$, there is an $m$ such that 
\[
\rm{F}^m\rm{H}^i(X, \F) \subset I^n\rm{H}^i(X, \F).
\]

Lemma~\ref{lemma:FP-approx-enough} ensures that it suffices to prove the claim under the assumption that $\F$ is a quasi-coherent $\O_X$-module of finite type. In particular, $\F$ is FP-approximated with respect to $I_0=(a_1, \dots, a_{r-1})$ and $a_r$ by Lemma~\ref{lemma:FP-approx-finite-type}(\ref{lemma:FP-approx-finite-type-1}). We also note that both pairs $(R, I_0)$ and $(R, a_r)$ are universally pseudo-adhesive. 

Indeed, \cite[Proposition 0.8.2.16]{FujKato} implies that they satisfy the $(\mathbf{BT})$ property, i.e. any finite $R$-module $M$ has bounded $a_r$-power and $I_0$-power torsion. Clearly, $\Spec R$ is noetherian outside $\rm{V}(a_r)$ and $\rm{V}(I_0)$ as $(a_r), I_0\subset I$. Applying the same argument to $R[T_1, \dots, T_d]$ for every $d$, we get that $R$ is universally $I_0$-adically and $a_r$-adically pseudo-adhesive. Therefore, the induction hypothesis can be applied to both $(R, I_0)$ and $(R, a_r)$.  \smallskip

\smallskip

Now we consider the short exact sequence
\[
0 \to a_r^k \F \to \F \to \F/a_r^k \F \to 0
\]
and define 
\[
\rm{H}^i\coloneqq \rm{Im}\left( \rm{H}^i(X, \F) \to \rm{H}^i(X, \F/a_r^k \F)\right)
\]
with the topology induced from the natural $I$-topology on $\rm{H}^i(X, \F)$. More precisely, it is topology defined by the filtration 
\[
\rm{F}^n\rm{H}^i\coloneqq \rm{Im}\left(\rm{F}^n\rm{H}^i(X, \F) \to \rm{H}^i\right)=\rm{Im}\left(\rm{H}^i(X, I^n\F) \to \rm{H}^i(X, \F/a^k_r\F)\right)
\]
The following two claims finish the proof. 

\begin{claim} It suffices to show that the topology on $\rm{H}^i$ coincides with the $I$-adic topology for any $k\geq 0$
\end{claim}
\begin{proof}
Step~$1$ justifies that there is $d$ such that $\rm{Im}\left(\rm{H}^i\left(X, a_r^d\F\right) \to \rm{H}^i\left(X, \F\right)\right) \subset a_r^n\rm{H}^i\left(X, \F\right)$. Then we use the assumption for $k=d$ to see that there is $m$ such that
\[
\rm{F}^m \rm{H}^i \subset I^n \rm{H}^i.
\]
This implies that 
\[
\rm{F}^m\rm{H}^i(X, \F) \subset I^n\rm{H}^i(X, \F) +  \rm{Im}\left(\rm{H}^i(X, a_r^d\F) \to \rm{H}^i\left(X, \F\right)\right) \subset I^n\rm{H}^i(X, \F) + a_r^n\rm{H}^i(X, \F) \subset I^n\rm{H}^i(X, \F).
\]
So this constructs the desired $m$. 
\end{proof}

\begin{claim} The topology on $\rm{H}^i$ coincides with the $I$-adic topology for any $k\geq 0$.
\end{claim}
\begin{proof}
Clearly, $\rm{F}^n\rm{H}^i \subset I^n\rm{H}^i$. Thus, we only need to show that, for any $n$, there is $m$ such that $\rm{F}^m\rm{H}^i \subset I^n\rm{H}^i$. Now we note that the $I_0$-adic topology on $\rm{H}^i$ coincides with the $I$-adic topology on $\rm{H}^i$. Therefore, it suffices to show that, for any $n$, there is $m$ such that 
\[
\rm{F}^m\rm{H}^i \subset I_0^n\rm{H}^i. 
\]
Now Theorem~\ref{thm:FP-approx-proper} $\rm{H}^i(X, \F/a^k_r\F)$ is an FP-approximated module for the pair $(R, I_0)$. Therefore, 
$\rm{H}^i$ is also FP-approximated as a submodule of an FP-approximated module $\rm{H}^i(X, \F/a^k_r\F)$. Now Lemma~\ref{lemma:fp-approx-subtopolgy} says that the subspace topology on $\rm{H}^i$ coincides with the $I_0$-adic topology. Thus, it suffices to show that, for any $n$, there is $m$ such that 
\[
\rm{F}^m\rm{H}^i \subset I_0^n\rm{H}^i\left(X, \F/a_r^k\F\right). \]
However, there is an evident inclusion
\[
\rm{F}^m\rm{H}^i \subset \rm{F}_{I_0}^m\rm{H}^i\left(X, \F/a_r^k\right).
\]
Now we invoke the induction hypotheses to say that the natural $I_0$-adic topology on $\rm{H}^i(X, \F/a_r^k\F)$ coincides with the $I_0$-adic topology in $\rm{H}^i(X, \F/a_r^k\F)$. This, in turn, implies that there is $m$ such that
\[
\rm{F}^m\rm{H}^i \subset \rm{F}_{I_0}^m\rm{H}^i\left(X, \F/a_r^k\F\right) \subset I_0^n \rm{H}^i\left(X, \F/a_r^k\F\right). 
\]
\end{proof}
\end{proof}

\bibliography{biblio}

\end{document}